%% file: main.tex
\documentclass{article}

\input{preamble.tex}
\begin{document}

\title{A geometric approach to Hall algebras I: Higher Associativity}
\author{Adam Gal, Elena Gal}
\maketitle

\begin{abstract}
\input{Abstract.tex}    
\end{abstract}

\tableofcontents

\input{Introduction.tex}
\input{Notations.tex}
\input{Waldhausen.tex}

\input{HallAlgebra.tex}
\input{Transfer.tex}
\input{Representations.tex}
\begin{appendices}
\input{PullbackCubes.tex}
\input{ProtoAbelian.tex}
\end{appendices}

\printbibliography
\end{document}

%% file: preamble.tex
\usepackage{amsmath}
\usepackage{amssymb}
\usepackage{amsthm}
\usepackage{comment}
\usepackage[backend=bibtex,style=alphabetic,citestyle=alphabetic]{biblatex}
\usepackage{mathtools}
\usepackage{hyperref}
\usepackage{bbm}
\usepackage[utf8]{inputenc}
\usepackage{tikz-cd}
\usepackage{wasysym}
\usepackage{varwidth}
\usepackage{mathtools}
\usepackage{xspace}
\usepackage{aliascnt}
\usepackage[bbgreekl]{mathbbol}
\usepackage[title,toc]{appendix}

\usepackage{titlesec}
\titleformat{\paragraph}[runin]{\normalfont\normalsize\itshape}{\theparagraph}{}{}[.] 
\titlespacing{\paragraph}{0pt}{0pt}{*1} 

\usepackage[2cell,ps]{xy}
\input xy
\UseAllTwocells
\xyoption{all}

\mathchardef\mhyphen="2D
\DeclareMathOperator{\Hom}{Hom}

\DeclareMathOperator{\pt}{pt}

\newcommand{\Func}{\mathcal{F}}

\DeclareMathOperator{\Ob}{Ob}
\DeclareMathOperator{\Exact}{Exact}

\DeclareMathOperator{\Rep}{Rep}
\newcommand{\Vect}{\mathbf{Vect}}

\DeclareMathOperator{\Id}{Id}

\DeclareMathOperator{\Spaces}{Stacks}
\newcommand{\Stacks}{\Spaces}
\DeclareMathOperator{\sset}{sSet}

\DeclareMathOperator{\Corr}{Corr}

\DeclareMathOperator{\Sh}{Sh}
\newcommand{\ShConst}{\Sh^{const}}

\newcommand{\GL}{GL}
\newcommand{\CCC}{\mathcal{C}}

\newcommand{\AAA}{\mathcal{A}}

\newcommand{\Icorr}{\widetilde{I}}

\DeclareMathOperator{\dgCat}{dgCat}
\DeclareMathOperator{\LinCat}{LinCat}
\newcommand{\aug}[1]{A(#1)}
\newcommand{\grid}{\operatorname{grid}}

\newcommand{\OrdSet}{\mathbb{\Delta}}
\newcommand{\OrdSetMod}{\mathbb{\Delta_{mod}}}
\newcommand{\AugOrdSet}{\mathbb{\Delta}_+}

\newcommand{\FF}{\mathbbm{F}}

\newcommand{\NN}{\mathbbm{N}}

\DeclareMathOperator{\point}{\pt}

\newcommand{\CC}{\mathbbm{C}}

\newcommand{\DDD}{\mathcal{D}}

\newcommand{\ord}[1]{\left[ #1 \right] }
\newcommand{\ordop}[1]{\Delta_{#1} }

\newcommand{\field}{\mathbbm{k}}

\newcommand{\tik}{\begin{tikzcd}}
\newcommand{\tak}{\end{tikzcd}}

\makeatletter
\def\latearrow#1#2#3#4{%
  \toks@\expandafter{\tikzcd@savedpaths\path[/tikz/commutative diagrams/every arrow,#1]}%
  \global\edef\tikzcd@savedpaths{%
    \the\toks@%
    (\tikzmatrixname-#2)
    to%
    node[/tikz/commutative diagrams/every label] {$#4$}
    (\tikzmatrixname-#3)
;}}
\makeatother

\def\stik#1#2{
\begin{tikzpicture}[baseline= (a).base]%
\node[scale=#1] (a) at (0,0){%
\begin{tikzcd}[ampersand replacement=\&]%
#2
\end{tikzcd}%
};%
\end{tikzpicture}
}

\newtheorem{Theorem}{Theorem}

\newtheorem{Proposition}{Proposition}[section]

\newaliascnt{Corollary}{Proposition}
\newtheorem{Corollary}[Corollary]{Corollary}
\aliascntresetthe{Corollary}

\newaliascnt{Lemma}{Proposition}
\newtheorem{Lemma}[Lemma]{Lemma}
\aliascntresetthe{Lemma}

\theoremstyle{definition}
\newaliascnt{Definition}{Proposition}
\newtheorem{Definition}[Definition]{Definition}
\aliascntresetthe{Definition}

\theoremstyle{remark}
\newaliascnt{Notation}{Proposition}
\newtheorem{Notation}[Notation]{Notation}
\aliascntresetthe{Notation}

\newaliascnt{Remark}{Proposition}
\newtheorem{Remark}[Remark]{Remark}
\aliascntresetthe{Remark}

\newaliascnt{Example}{Proposition}
\newtheorem{Example}[Example]{Example}
\aliascntresetthe{Example}

\def\equationautorefname~#1\null{%
(#1)\null
}

%% file: Abstract.tex
We construct a geometric system from which the Hall algebra can be recovered. This system inherently satisfies higher associativity conditions and thus leads to a categorification of the Hall algebra. We then suggest how to use this approach to construct categorified representations.

%% file: Introduction.tex
\section{Introduction}
To an abelian category $\CCC$ (with appropriate restrictions) one can assign an algebra $A$ called the \emph{Hall algebra} of $\CCC$. The first major application was in \cite{ringel1990hall} where Ringel showed that the positive half of the quantum group is isomorphic to the Hall algebra of the category $\Rep_{\FF_q}(Q)$ for $Q$ the quiver corresponding to the Dynkin diagram.

\begin{Example}
Take $Q$ to be the quiver with one point, then $\Rep_{\FF_q}(Q)$ is just $\Vect_{\FF_q}$. The Hall algebra $A$ has basis $\mathbf{n},n\in\NN$ and multiplication given by
\[
\mathbf{n}\cdot\mathbf{m}=\#\left(\{0\to\FF_q^n\to\FF_q^{n+m}\to\FF_q^m\to 0\}/\sim\right)\cdot(\mathbf{n+m})
\]
Even in this simplest example we see that geometry is evident in the Hall algebra, as the coefficient appearing is obviously the number of points of an algebraic variety (a Grassmannian in this case).
\end{Example}

Luzstig used this observation in \cite{lusztig1991quivers} to construct a canonical basis for the Hall algebra (and hence for the positive halves of quantum groups) and to prove the positivity for the canonical basis.

In this article we describe a coherent construction of a system of geometric objects which govern the Hall algebra, and which have inherent in them a higher  associative structure. 

In \cite{KapranovDyckerhoff} the authors consider the connection between Hall algebras and the Waldhausen construction. They introduce the notion of 2-Segal space and show that the Waldhausen construction for an abelian category $\CCC$ is a 2-Segal space. This fact in particular imples the associativity of the Hall algebra. Moreover they indicate the connection between the 2-Segal conditions and the higher associativity constraints. We explore this subject systematically in the present work. In \autoref{Corr0} we state the connection between the 2-Segal conditions and the higher associativity data we construct.

In \cite{KapranovDyckerhoff} the authors suggest to study Hall algebras on the level of category of correspondences. Consider the groupoids (or stacks) of objects of a category $\Ob$ and exact sequences $\Exact$. Then the correspondence

\[
\stik{1}{
{} \& \Exact \ar{dl}[above, xshift=-0.5em]{end} \ar{dr}{mid} \& {} \\
\Ob\times \Ob \& {} \& \Ob
}
\]
defines a monoidal structure on $\Ob$. Using a composition of appropriate push-and pull- functors this correspondence can be used to define multiplication in the Hall algebra of (finitely supported) functions on groupoids and it's geometric analog constructed using sheaves.

We construct a coherent system of higher associativity data for this structure. This data is provided by an explicit system of $n$-cubes of correspondences for any $n$. Altogether this construction organizes into a functor $\AugOrdSet\rightarrow Corr(\Spaces)$. In our forthcoming article \cite{ourGeometricHall2} we extend this to a functor from $\AugOrdSet\otimes\AugOrdSet$ which captures a bi-monoidal structure associated to Hall algebras. Such a construction was attempted also in \cite{Lyubashenko} from a different point of view.

Several interesting examples of Hall algebras can then be obtained from this abstract setting using so called \emph{transfer theories}. The examples treated in \cite{KapranovDyckerhoff} include Ringel's Hall algebra associated to the category of representations of a quiver, Lusztig’s geometric Hall algebra, To\"en's derived Hall algebra, Joyce’s motivic Hall algebra, and Kontsevich-Soibelman’s cohomological Hall algebra. 

In this article we outline a transfer theory to the 
category of constructible sheaves. The difference from the examples listed above is that we obtain a monoidal 
category as opposed to an algebra and thus have to make use of the higher associators mentioned above.

We show that the 2-Segal conditions of \cite{KapranovDyckerhoff} are equivalent to our associators going to isomorphisms under this transfer construction.

Finally, we suggest an approach to constructing categorified representations in \autoref{Representations}.

In what follows we will assume for the ease of presentation that we work with finitary abelian categories. However the construction we describe can be straightforwardly generalized to different settings, such as in the list of examples above.

%% file: Notations.tex
\section{Notations}
$\OrdSet$ - the category of finite ordered sets. The elements of $\OrdSet$ will be denoted by \[\ordop{0}=0, \ordop{1}=0\rightarrow 1, \ordop{2}=0\rightarrow 1\rightarrow 2, \ldots\]
$\AugOrdSet$ - the augmented category of finite ordered sets. The elements of $\AugOrdSet$ will be denoted by \[\ord{0}=\emptyset, \ord{1}=0, \ord{2}=0\rightarrow 1, \ord{3}=0\rightarrow 1\rightarrow 2, \ldots\]


%% file: Waldhausen.tex
\section{The Waldhausen S-construction}
\label{Waldhausen}
Let $\CCC$ be an abelian category. We recall the construction of a simplicial system of spaces associated to $\CCC$ called the Waldhausen construction. This system of spaces plays a role in defining multiplication in the Hall algebra of $\CCC$ and providing the associator data. Our exposition in this section follows \cite{KapranovDyckerhoff}. 

Note that the classes of mono- and epimorphisms in $\CCC$ satisfy the following: 
 \begin{enumerate}
 \item Any commutative square with monomorphisms as vertical and epimorphisms as horizontal maps is a pullback iff it is a pushout. We will call such squares bicartesian.
  \item Pullbacks and pushouts of monomorphisms along epimorphisms exist.
 \end{enumerate}

The Waldhausen construction assigns to  a linear category $\CCC$ a functor $S_{-}\CCC:\OrdSet^{op}\to \Spaces$ as follows:

\begin{enumerate}
    \item Take $X\in\OrdSet$
    \item Consider $\grid(X):=\Hom(0\to 1,X)$, as a marked category, with the constant maps the marked objects. Note that $\grid(X)$ has two classes of maps, the "horizontal" and the "vertical", by which we mean maps that are identity on the 0 or 1 component, respectively. 
    \item Define $S_X\CCC$ to be the stack of maps from $\grid(X)$ to $\CCC$ which take the marked objects to $0$, the horizontal maps to monomorphisms and the vertical maps to epimorphisms, and take Cartesian squares to Cartesian squares.
\end{enumerate}


\begin{Remark}
The last stage can be streamlined by noting that (assuming some restrictions on $X$) $\grid(X)$ has a natural structure of an \emph{extended} proto-abelian category, i.e. instead of one zero object it has a "zero subcategory" - the constant maps. Then $S_X\CCC$ can be considered to be just maps of proto abelian categories in this wider sense. See \autoref{ProtoAbelian}
\end{Remark}

\begin{Example}
Take $X=\ordop{1}=0\to 1$, then \[
\grid(X)=\stik{1}{
00 \ar{r} \& 01 \ar{d} \\
{} \& 11
}
\]
and so $S_X\CCC=S_{\ordop{1}}\CCC$ is the space of objects of $\CCC$. 
\end{Example}

\begin{Example}
Take $X=\ordop{2}=0\to 1\to 2$, then  \[
\grid(X)=\stik{1}{
00 \ar{r} \& 01 \ar{r} \ar{d} \& 02 \ar{d} \\
{} \& 11 \ar{r} \& 12 \ar{d} \\
{} \& {} \& 22
}
\]
The data of a map $\grid(X)\to\CCC$ then consists of a square\[
\stik{1}{
C_{01} \ar[hook]{r} \ar[two heads]{d} \& C_{02} \ar[two heads]{d}\\
C_{11}=0 \ar[hook]{r} \& C_{12}
}
\]
which must be Cartesian and therefore also coCartesian. This just means that $C_{01}\to C_{02} \to C_{12}$ is an exact sequence.
In all, $S_X\CCC=S_{\ordop{2}}\CCC$ is the stack of exact sequences in $\CCC$.
\end{Example}
It is now easy to guess the general shape of $S_X\CCC$. Namely, for $X=\ordop{n}$ we get the diagrams of the form 
\[
\stik{1}{
0 \ar[hook]{r} \& C_{01} \ar[two heads]{d} \ar[hook]{r} \& C_{02}\ar[two heads]{d} \ar[hook]{r} \& \cdots  \ar[hook]{r} \& C_{0n}\ar[two heads]{d} \\
{} \& 0  \ar[hook]{r} \& C_{12}\ar[two heads]{d} \ar[hook]{r} \& \cdots  \ar[hook]{r} \& C_{1n}\ar[two heads]{d} \\
{} \& {} \& 0 \ar[hook]{r} \& \cdots  \ar[hook]{r} \& C_{2n}\ar[two heads]{d}\\
{} \& {} \& {} \& \ddots \& \vdots\ar[two heads]{d} \\
{} \& {} \& {} \& {} \& 0
}
\]

where every square is biCartesian.

It is shown in \cite{KapranovDyckerhoff} Lemma 2.4.9 that the groupoid of diagrams of this shape is equivalent to the groupoid of flags of length $n$ providing the connection to the classical Waldhausen construction. Their argument can be generalized to stacks in a straightforward manner.

In the following section we will extend the Waldhausen construction to a functor from a certain subcategory of $\sset$ to $\Spaces$.

%% file: HallAlgebra.tex
\section{The geometric Hall algebra}
\label{geoHall}
\begin{Notation}
From now on we will shorten $S_{\ordop{n}}\CCC$ to $S_{n}$.
\end{Notation}

\subsection{The associative multiplication}
The spaces $S_1$ and $S_2$ of the Waldhausen construction are used to define multiplication in the Hall algebra. Recall that $S_1$ is the stack of objects of $\CCC$ and $S_2$ is the stack of short exact sequences. We have the maps $mid:S_2 \to S_1$ and $end:S_2 \to S_1\times S_1$ given by 
\begin{align*}
    mid(0\to U\to V\to W\to 0)&=V\\
    end(0\to U\to V\to W\to 0)&=(U,W)
\end{align*}

This gives us a correspondence 
\[
\stik{1}{
{} \& S_2 \ar{dl}[above, xshift=-0.5em]{end} \ar{dr}{mid} \& {} \\
S_1\times S_1 \& {} \& S_1
}
\]

Our objective is to explicitly present the associativity data for this multiplication. The usual approach to showing the associativity of multiplication on the 1-categorical level comes down to considering the associativity square, i.e. the square 
\[
\stik{1}{
S_1^3 \& \& S_2\times S_1 \ar{ll} \ar{rr} \& \& S_1\times S_1 \\
S_1\times S_2 \ar{u} \ar{d} \& P_2 \ar{dr} \ar{l} \& {} \& P_1 \ar{ul} \ar{r}\&  S_2 \ar{u} \ar{d} \\
S_1\times S_1 \& \& S_2 \ar{ll} \ar{rr} \& \& S_1
}
\]
To show that this square commutes in the 1-category of correspondences one takes pullbacks $P_1$ and $P_2$ corresponding to compositions of the upper and right and left and lower sides in the category of correspondences and shows that they are isomorphic. In \cite{KapranovDyckerhoff} it is shown using relations between $S_1,S_2,S_3$ which are an instance of what the authors call the 2-Segal conditions. 

However if we want to consider the higher associativity data we have to adapt a more systematic approach. For instance in this square we will have an explicitly specified middle term

\[
\stik{1}{
S_1^3 \& S_2\times S_1 \ar{l} \ar{r} \& S_1\times S_1 \\
S_1\times S_2 \ar{u} \ar{d} \& X\ar{u} \ar{d} \ar{l} \ar{r} \&  S_2 \ar{u} \ar{d} \\
S_1\times S_1 \& S_2 \ar{l} \ar{r} \& S_1
}
\]
Where the appropriate squares are pullbacks. In this case $X=S_3$, and this condition recovers the 2-Segal condition.

We will explicitly construct a system of $n$-cubes serving as candidates for higher associativity data. That this data provides higher isomorphisms in the category of correspondences is equivalent to the 2-Segal conditions of \cite{KapranovDyckerhoff} (see \autoref{2Segalpullbacks})


\subsection{\texorpdfstring{$n$}{n}-cubes of correspondences}
\label{Corr}

We describe what we mean by an $n$-cube in the category of correspondences. For $\DDD$ a category we take $\Corr(\DDD)$ to be the system of cubes (formally a \emph{cubical set}) with points being objects of $\DDD$ and arrows being correspondences (spans) in $\DDD$
\[
\stik{1}{
A \& E \ar{l} \ar{r} \& B
}
\]
A 2-cube is a diagram of the form
\begin{equation}
\label{CorrSquare}
\stik{1}{
A \& E \ar{l} \ar{r} \& B \\
F \ar{u} \ar{d}\& Z \ar[red,thick]{u} \ar[red,thick]{d}\ar[red,thick]{l} \ar[red,thick]{r} \& G\ar{u} \ar{d} \\
C \& H \ar{l} \ar{r} \& D \\
}
\end{equation}
We will say that a 2-cube \emph{commutes} if the upper right and lower left squares are Cartesian. This is equivalent to the usual notion in the 1-category of correspondences.

A 3-cube is a diagram 
\begin{equation*}
\label{CorrCube}
\stik{0.7}{
{} \& \& \phantom{XX}A_3 \& \& \phantom{XX}E_3 \ar[xshift=2em,shorten >=-2em]{ll} \ar{rr} \& \& B_3\\
\\
{} \& A_2 \ar{uur} \ar{ddl} \& \& E_2  \ar{uur} \ar{ddl}  \ar{ll} \ar{rr} \& \& B_2  \ar{uur} \ar{ddl} \\
{} \& {}\& \phantom{XX}F_3 \ar[xshift=1em]{uuu} \ar[xshift=1em]{dddd}\& {} \& \phantom{XX}Z_3 \ar{rr} \ar[xshift=2em,shorten >=-2em]{ll} \ar[xshift=1em]{uuu} \ar[xshift=1em]{dddd}\& {} \& G_3 \ar{uuu} \ar{dddd}\\
A_1 \& \& E_1 \ar{ll} \ar{rr} \& \& B_1 \& \&\\
\\
{} \& F_2  \ar{uuur} \ar{ddl} \ar{uuuu} \ar{dddd}\& \& Z_2  \ar[red,thick]{uuur} \ar[red,thick]{ddl} \ar[red,thick]{uuuu} \ar[red,thick]{dddd}\ar[red,thick]{ll} \ar[red,thick]{rr} \& \& G_2 \ar{uuur} \ar{ddl}  \ar{uuuu} \ar{dddd}\\
{} \& {}\&\phantom{XX}C_3 \& {}\&\phantom{XX}H_3 \ar{rr} \ar[xshift=2em,shorten >=-2em]{ll} \& {}\& D_3\\
F_1 \ar{uuuu} \ar{dddd}\& \& Z_1 \ar{uuuu} \ar{dddd}\ar{ll} \ar{rr} \& \& G_1\ar{uuuu} \ar{dddd} \& \& \\
\\
{} \& C_2  \ar{uuur} \ar{ddl} \& \& H_2 \ar{uuur}\ar{ddl} \ar{ll} \ar{rr} \& \& D_2  \ar{uuur} \ar{ddl} \\
\\
C_1 \& \& H_1 \ar{ll} \ar{rr} \& \& D_1 
}
\end{equation*}

We will say that a cube of correspondences commutes if the cubes in the upper-right-back and lower-left-front corners are \emph{pullback cubes} (see \autoref{PullbackCube}). This coincides with the usual notion of commutativity of a cube in a 2-category in the same way as above for squares.

\begin{Definition}
\label{commutativecubes}
Representing the vertices in an $n$-cube by sequences of 0's and 1's, we can now define the higher commutative cubes by the property that the two $n$-cubes generated by all faces containing the vertex $(1,0,1,0,\ldots)$ or all faces containing the vertex $(0,1,0,1,\ldots)$ are pullbacks. 
\end{Definition}

\subsection{The combinatorial construction}
\label{Hcomb}
Our construction of the associativity data splits into two parts. We first construct the system of $n$-cubes in the category $\Corr(\sset^{op})$. Next we apply the extension of the Waldhausen construction described in \autoref{Waldhausen} $S^{ext}$ object-wise to obtain a system of $n$-cubes in $\Corr(\Spaces)$. 

We obtain a functor $H_{comb}: \AugOrdSet \rightarrow \Corr(SSet)$. By composing with the extended Waldhausen construction we then obtain a functor $H_{geo}: \AugOrdSet \rightarrow \Corr(Stacks)$. 

Formally the most straightforward way is to consider these functors as functors between categories modelled by cubical sets. However we do not wish to enter into technical details and limit ourselves to the very explicit construction which is the objective of this article. This construction is obviously functorial in any conceivable sense.

We start from the combinatorial construction. For every $n$-cube in $\AugOrdSet$ we construct a corresponding  $n$-cube in $\Corr(\sset^{op})$. 

\begin{Definition}
For $X\in\AugOrdSet$ define the \emph{augmentation} of $X$, $\aug{X}$ to be $\Hom_\OrdSet(X,0\rightarrow 1)$ considered as an object in $\OrdSet^{op}$.
\end{Definition}

This sends the totally ordered set with $n$ elements in $\AugOrdSet$ (which we denote $\ord{n}$) to the totally ordered set with $n+1$ elements in $\OrdSet^{op}$ (which we denote $\ordop{n}$).

\subsubsection{Construction for objects}
\label{HcombObjects}
Let $X\in\AugOrdSet$. Each element $x\in X$ determines an embedding of $\aug{\{x\}}$ in $\aug{X}$ as follows: Given a map $\{x\} \to \{0\rightarrow 1\}$, we extend it to a map $X \to \{0 \rightarrow 1\}$ by setting it to $0$ on lower than $x$ elements and $1$ on higher than $x$ elements of $X$. Consider the sub-simplicial set of $\aug{X}$ (considered as an object in $\sset^{op}$) generated by these embeddings. We will denote this simplicial set by $H_{comb}(X)$.

\begin{Example}
The first few values of $H_{comb}$ on the objects of $\AugOrdSet$ are as follows
\begin{itemize}
\item $H_{comb}(\ord{0})=\ordop{0}$
\item $H_{comb}(\ord{1})=\ordop{1}$ 
\item $H_{comb}(\ord{2})$ is the horn 
\includegraphics[scale=0.3]{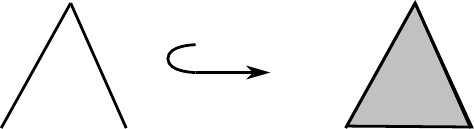}
\item $H_{comb}(\ord{3})$ is \includegraphics[scale=0.3]{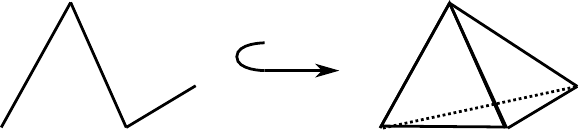}
\end{itemize}
\end{Example}
From the description in \autoref{Waldhausen} it is clear how $S$ extends to these objects. Namely, we have 
\begin{itemize}
\item $\ordop{0}\xmapsto{S^{ext}}S_0(C)=\point$
\item $\ordop{1}\xmapsto{S^{ext}}S_1(\CCC)$ 
\item $\includegraphics[scale=0.3]{Figures/2HornComb.pdf} \xmapsto{S^{ext}}S_1(\CCC)\times S_1(\CCC)\hookrightarrow S_2(\CCC)$
\item $\includegraphics[scale=0.3]{Figures/3HornComb.pdf}\xmapsto{S^{ext}}S_1(\CCC)\times S_1(\CCC)\times S_1(\CCC) \hookrightarrow S_3(\CCC)$
\end{itemize}

\subsubsection{Construction for arrows}

Let $X \xrightarrow{f} Y$ be a map in $\AugOrdSet$. We want to associate to it a correspondence in $\sset^{op}$.

Let $y\in Y$ and denote $X_y$ the preimage of $y$ under $f$. Similarly to the above, we get an imbedding $\aug{X_y}$ in $\aug{X}$. Denote the sub-simplicial set generated by these imbeddings for all $y$ by $H_{comb}(f)$. Then we have a natural correspondence \[
H_{comb}(X) \rightarrow{} H_{comb}(f) \leftarrow H_{comb}(Y)
\]

\begin{Example}
The multiplication is the image of the map $\ord{2} \to \ord{1}$, and on the level of $H_{comb}$ this map goes to 
\[
\includegraphics[scale=0.5]{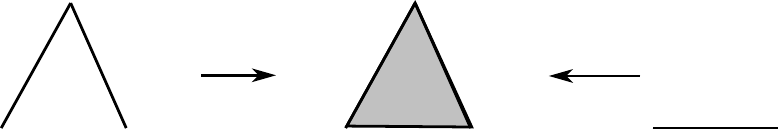}
\]
$S^{ext}$ then sends the middle object to the short exact sequences, and the maps to restriction to the endpoints or middle respectively. This is the correspondence defining the multiplication in the Hall algebra.
\end{Example}


\subsubsection{Construction for squares}

Given a square in $\OrdSet$
\[
\stik{1}{
X \ar{r}{f} \ar{d}{g} \& Y \ar{d}{h}\\
Z \ar{r}{k} \& W
}
\]
Similarly to what we did with arrows, we have a map $X\xrightarrow{\alpha}W$ and we then construct the square of correspondences

\[
\stik{1}{
H_{comb}(X) \ar{r} \ar{d} \& H_{comb}(f) \ar{d}{}\& H_{comb}(Y) \ar{l} \ar{d} \\
H_{comb}(g) \ar{r} \& H_{comb}(\alpha) \& H_{comb}(h) \ar{l}{} \\
H_{comb}(Z) \ar{u} \ar{r} \& H_{comb}(k) \ar{u} \& H_{comb}(W) \ar{l} \ar{u}
}
\]

where $H_{comb}(\alpha)$ is the sub simplicial set in $\aug{X}$ generated by the imbeddings of $\aug{(h\circ f)^{-1}(w)}$ for all $w\in W$.

The construction for general $n$-cubes following this example is straightforward and we omit the tedious general definition. See \autoref{2Segalpullbacks} for more examples.

\subsection{The higher associativity cubes}
The higher associators in the Hall algebra are given by the images of the following family of commutative cubes in $\AugOrdSet$:

\begin{Lemma}
For any $n\geq 3$ there is a unique commutative $n-1$ cube in $\AugOrdSet$ that contains all of the surjections $\ord{j}\to\ord{j-1}$ for all $1\leq j\leq n$.
\end{Lemma}
We call such cube "the $n$-associativity cube. The paths from $\ord{n}$ to $\ord{1}$ on the $n$-associativity cube correspond exactly to all ways of bracketing $n$ letters.

\begin{Example}
For $n=3$ the $n$-associativity cube is the square
\[
\stik{1}{
\ord{3} \ar[two heads]{r} \ar[two heads]{d} \& \ord{2} \ar[two heads]{d} \\
\ord{2} \ar[two heads]{r} \& \ord{1}
} \leftrightarrow \left(\alpha: (XY)Z\to X(YZ)\right) 
\]
\end{Example}

\begin{Example}
For $n=4$ the $n$-associativity cube is the cube 
\begin{gather*}  
\stik{1}{
 \& \ord{3} \ar[two heads]{rr} \ar[two heads]{dd} \& {} \& \ord{2} \ar[two heads]{dd}\\
\ord{4} \ar[two heads,crossing over]{rr} \ar[two heads]{dd} \ar[two heads]{ur} \& {} \& \ord{3} \ar[two heads]{dd} \ar[two heads]{ur} \& {}\\
{} \& \ord{2} \ar[two heads]{rr} \& {} \& \ord{1}\\
\ord{3} \ar[two heads]{rr} \ar[two heads]{ur} \& {} \& \ord{2} \ar[two heads]{ur}
\latearrow{commutative diagrams/crossing over,commutative diagrams/two heads}{2-3}{4-3}{}
}
\\
\updownarrow\\
\stik{1}{
{} \& (X(YZ))W \ar{r}{\alpha_{{}_{X,Y\cdot Z,W}}}\&  X((YZ)W) \ar{dr}{X\cdot \alpha}\\
((XY)Z)W \ar{ur}{\alpha\cdot W} \ar{dr}{\alpha_{{}_{X\cdot Y,Z,W}}} \& {} \&  {} \& X(Y(ZW))\\
{} \& (XY)(ZW) \phantom{X}\ar[shorten <=-1em,shorten >=-1em]{r}{\Id_{XY}\cdot\Id_{ZW}} \& \phantom{X}(XY)(ZW) \ar{ur}{\alpha_{{}_{X,Y,Z\cdot W}}}
}
\end{gather*}

The image of this cube under $H_{comb}$ and the extended Waldhausen construction is a cube of correspondences of stacks, where each face has a 2-morphism specified by $S$. These 2-morphisms compose along the edges and form a diagram whose commutativity is equivalent to the pentagon identity (This diagram is the pentagon diagram, with one of the edges equal $\Id$).

\end{Example}
The higher associativity cubes correspond to higher associators.






\subsection{Recovering the 2-Segal conditions}
This section contains the proof of the following:
\label{2Segalpullbacks}

\begin{Theorem}
\label{Corr0}
The 2-Segal conditions satisfied by the original Waldhausen construction are equivalent to the requirement that the functor $H_{geo}:\AugOrdSet \rightarrow \Corr(\Stacks)$ sends the higher associativity cubes to commutative cubes in $\Corr(\Stacks)$.
\end{Theorem}
Let us recall the 2-Segal condition from \cite[\S 2.3]{KapranovDyckerhoff}:

Let $S:\OrdSet\to\Stacks$ be a functor, and $P$ a polygonal decomposition of an $n$-gon, written as $(P_1,\ldots,P_k)$. e.g.
\[
\input{Figures/PolygonalDecomp.tex}
\]

We define $S_P$ to be $S_{P_1}\times_{S_{P_1\cap P_2}} S_{P_2}\times_{S_{P_2\cap P_3}} S_{P_3}\ldots \times_{S_{P_{k-1}\cap P_k}} S_{P_k}$ where $S_{P_i}$ is $S_{\ordop{\#\{\text{vertices of }P_i\}}}$ and note that by functoriality we have a natural map
$S_{\ordop{n}}\to S_P$.

\begin{Definition}
A functor $S:\OrdSet\to\Stacks$ is said to be a 2-Segal stack if for and $n\geq 3$ and any polygonal decomposition $P$ of an $n$-gon the map $S_{\ordop{n}}\to S_P$ is a weak equivalence.
\end{Definition}

For our purposes we need the following dual version of Proposition 2.3.2 from \cite{KapranovDyckerhoff}:

\begin{Proposition}
The following are equivalent

\begin{enumerate}
\item $S$ is a 2-Segal stack.
\item For any $n\geq 3$ and any two disjoint subsets $I,J$ of $\ordop{n}$ that don't cover all the vertices of $\ordop{n}$ the following square
\[
\stik{1}{
S_{\ordop{n}} \ar{r} \ar{d} \& S_{\ordop{n}\setminus I} \ar{d}\\
S_{\ordop{n}\setminus J} \ar{r} \& S_{\ordop{n}\setminus (I\cup J)}
}
\]
is a pullback.
\item For any $n\geq 3$ and any $0\leq i < j\leq n$ of $\ordop{n}$ the following square
\[
\stik{1}{
S_{\ordop{n}} \ar{r} \ar{d} \& S_{\ordop{n}\setminus \{i\}} \ar{d}\\
S_{\ordop{n}\setminus \{j\}} \ar{r} \& S_{\ordop{n}\setminus \{i,j\}}
}
\]
is a pullback.
\item For any $n\geq 3$ and any $0\leq i\leq n-2$ of $\ordop{n}$ the following square
\[
\stik{1}{
S_{\ordop{n}} \ar{r} \ar{d} \& S_{\ordop{n}\setminus \{i\}} \ar{d}\\
S_{\ordop{n}\setminus \{i+2\}} \ar{r} \& S_{\ordop{n}\setminus \{i,i+2\}}
}
\]
is a pullback.
\end{enumerate}
\end{Proposition}
Denote by $C^n_i$ the conditions in part $(4)$.

\begin{proof}
Obviously $(2)\Rightarrow (3) \Rightarrow (4)$. $(1)\Rightarrow (2)$ because $(2)$ is equivalent to a special case of $(1)$ for the decomposition corresponding to \[\ordop{n}=I\cup J\cup (\ordop{n}\setminus (I\cup J))\]
It is also easy to see that any polygonal decomposition can be built up from such decompositions so $(2)\Rightarrow (1)$. $(3)\Rightarrow (2)$ similarly because we can remove the points of $I,J$ one by one.

$(4)\Rightarrow (1),(2),(3)$: Consider
\[
\stik{1}{
S_{\ordop{n}} \ar{r} \ar{d} \& S_{\ordop{n}\setminus \{i\}} \ar{d} \ar{r}  \& S_{\{i+1,i+2,i+3\}} \ar{d}\\
S_{\ordop{n}\setminus \{i+2\}} \ar{r} \& S_{\ordop{n}\setminus \{i,i+2\}} \ar{r} \& S_{\{i+1,i+3\}} 
}
\]
The right square is a pullback by induction on $n$, and the left square is a pullback by assumption, so the whole square is a pullback. In terms of polygonal decompositions this means we can separate away any triangle from the polygon and so we can get to any triangulation. Since any polygonal decomposition can be refined to a triangulation, we are done.

\end{proof}

\begin{proof}[Proof of \autoref{Corr0}]
Let us first consider the case $n=3$. The associativity square is \[
\stik{1}{
\ord{3} \ar{r} \ar{d} \& \ord{2} \ar{d} \\
\ord{2} \ar{r} \& \ord{1}
}
\]

Its image under $H_{comb}$ is
\[
\includegraphics[scale=0.8]{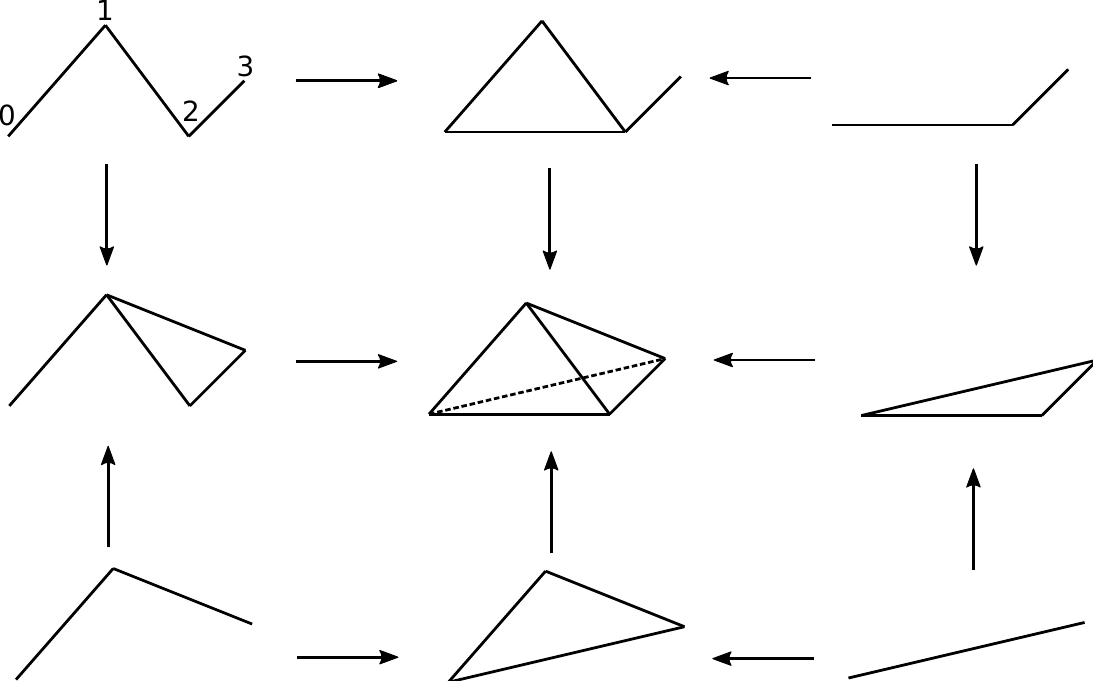}
\]

In order for this square to go to a commutative square of correspondences of stacks (see \autoref{commutativecubes}), $S^{ext}$ should take the squares in the upper right (i.e. $(1,0)$) and lower left (i.e. $(0,1)$) corners to pullback squares. This is easily seen to be equivalent to conditions $C^3_1$ and $C^3_0$.

Now consider $n=4$. The associativity cube is
\begin{equation}
\label{PentagonCube}
\stik{1}{
 \& \ord{3} \ar{rr} \ar{dd} \& {} \& \ord{2} \ar{dd}\\
\ord{4} \ar[crossing over]{rr} \ar{dd} \ar{ur} \& {} \& \ord{3} \ar{dd} \ar{ur} \& {}\\
{} \& \ord{2} \ar{rr} \& {} \& \ord{1}\\
\ord{3} \ar{rr} \ar{ur} \& {} \& \ord{2} \ar{ur}
\latearrow{commutative diagrams/crossing over}{2-3}{4-3}{}
}
\end{equation}

Its image under $H_{comb}$ is a cube of correspondences of $\sset^{op}$, that is, a $2\times 2\times 2$ grid of commutative cubes in $\sset^{op}$ so that the outer shell is comprised of the images of the faces of the cube \autoref{PentagonCube} and the center is the 4-simplex. 

In order for this cube to be commutative by \autoref{commutativecubes} the cubes in the upper-right-back (i.e. $(1,0,1)$) and lower-left-front (i.e. $(0,1,0)$) corners should go to pullback cubes under $S^{ext}$. Let's consider first the upper-right-back cube:
\[
\includegraphics[scale=1.2]{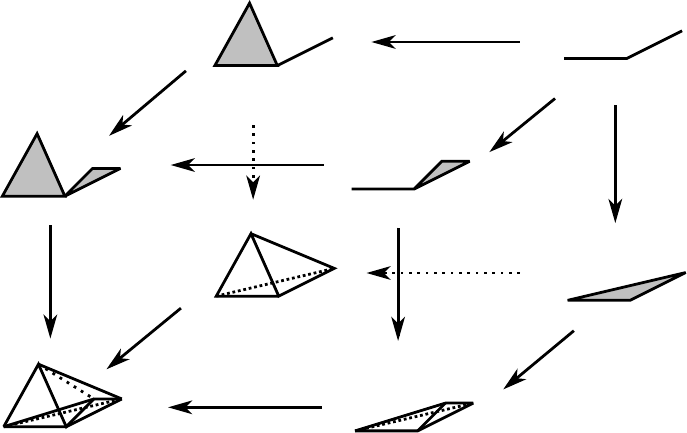}
\]

Its top face goes to a product of trivial squares, hence a pullback. Therefore by \autoref{pullbackcubeCor} the cube is a pullback iff the bottom face is a pullback, and this is exactly $C^4_1$.

Now consider the lower-left-front cube:
\[
\includegraphics[scale=1.2]{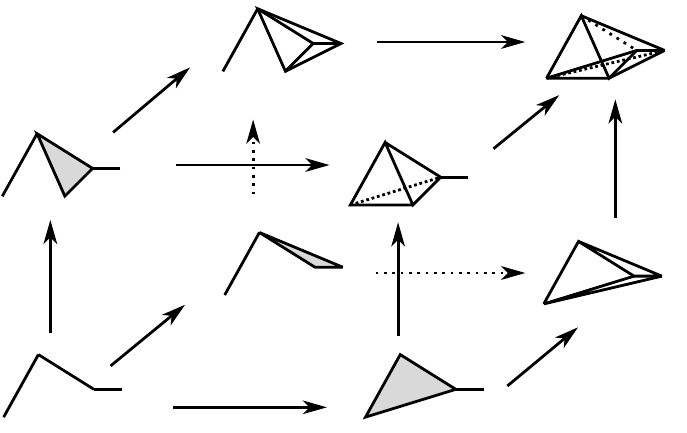}
\]

Its left and front faces correspond to conditions $C^3_1,C^3_2$ which follow from the previous case, and so for the cube to be a pullback we need either the right or the back faces to be pullbacks (in which case the other is as well). These give us conditions $C^4_2,C^4_0$.

We continue by induction, noting that sub-cubes of associator cubes are disjoint unions of lower dimensional associator cubes. The conditions $C^n_{2k}$ are recovered from the cube in the $(0,1,0,1,\ldots)$ corner being a pullback cube and the conditions $C^n_{2k+1}$ are recovered from the cube in the $(1,0,1,0,\ldots)$ corner being a pullback cube.
\end{proof}

%% file: Figures/PolygonalDecomp.tex
\begingroup%
  \makeatletter%
  \providecommand\color[2][]{%
    \errmessage{(Inkscape) Color is used for the text in Inkscape, but the package 'color.sty' is not loaded}%
    \renewcommand\color[2][]{}%
  }%
  \providecommand\transparent[1]{%
    \errmessage{(Inkscape) Transparency is used (non-zero) for the text in Inkscape, but the package 'transparent.sty' is not loaded}%
    \renewcommand\transparent[1]{}%
  }%
  \providecommand\rotatebox[2]{#2}%
  \ifx\svgwidth\undefined%
    \setlength{\unitlength}{156.92625631bp}%
    \ifx\svgscale\undefined%
      \relax%
    \else%
      \setlength{\unitlength}{\unitlength * \real{\svgscale}}%
    \fi%
  \else%
    \setlength{\unitlength}{\svgwidth}%
  \fi%
  \global\let\svgwidth\undefined%
  \global\let\svgscale\undefined%
  \makeatother%
  \begin{picture}(1,0.75123049)%
    \put(0,0){\includegraphics[width=\unitlength,page=1]{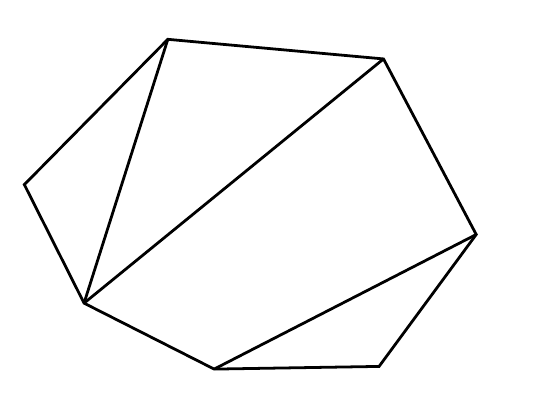}}%
    \put(0.12172059,0.40007091){\color[rgb]{0,0,0}\makebox(0,0)[lb]{\smash{}}}%
    \put(0.08573423,0.39352797){\color[rgb]{0,0,0}\makebox(0,0)[lb]{\smash{$P_1$}}}%
    \put(0.32618869,0.52438755){\color[rgb]{0,0,0}\makebox(0,0)[lb]{\smash{$P_2$}}}%
    \put(0.44069086,0.29701911){\color[rgb]{0,0,0}\makebox(0,0)[lb]{\smash{$P_3$}}}%
    \put(0.5748219,0.1203587){\color[rgb]{0,0,0}\makebox(0,0)[lb]{\smash{$P_4$}}}%
    \put(-0.0042317,0.41315677){\color[rgb]{0,0,0}\makebox(0,0)[lb]{\smash{$0$}}}%
    \put(0.26893765,0.71249808){\color[rgb]{0,0,0}\makebox(0,0)[lb]{\smash{$1$}}}%
    \put(0.7073172,0.67651165){\color[rgb]{0,0,0}\makebox(0,0)[lb]{\smash{$2$}}}%
    \put(0.90687802,0.32482657){\color[rgb]{0,0,0}\makebox(0,0)[lb]{\smash{$3$}}}%
    \put(0.67460231,0.02712161){\color[rgb]{0,0,0}\makebox(0,0)[lb]{\smash{$4$}}}%
    \put(0.3327317,0.0074925){\color[rgb]{0,0,0}\makebox(0,0)[lb]{\smash{$5$}}}%
    \put(0.06937679,0.14489469){\color[rgb]{0,0,0}\makebox(0,0)[lb]{\smash{$6$}}}%
  \end{picture}%
\endgroup%

%% file: Transfer.tex
\section{Transfer theories}
\label{Transfer}
In \autoref{geoHall} we describe a system of stacks $H_{geo}$ associated to a category $\CCC$. In order to recover from it any kind of algebra structure, we need to transfer the higher associative structure from correspondences of stacks to an algebraic setting $\AAA$ - formally a monoidal $\infty$ category with duals. We call such a gadget (following \cite{KapranovDyckerhoff}) a \emph{transfer theory}.

A transfer theory $T$ will need to assign to each $n$-cube of commutative correspondences an $n$-cube in our algebraic setting of choice.

To begin with, we need an assignment $X\mapsto T(X)$ on objects. Then we need for any correspondence $X\xleftarrow{f}Z\xrightarrow{g}Y$ a morphism $T(X)\to T(Y)$. Instead we will only require a to have an assignment $T(X)\xleftarrow{T(f)}Z\xrightarrow{T(g)}T(Y)$. Since by assumption $\AAA$ has duals we can choose a dual for $T(f)$ to get an actual morphism but we need not make a consistent choice of duals.

For squares, note that correspondences can be considered as $\Icorr$-diagrams for $\Icorr=(0\leftarrow M\rightarrow 1)$. Consider now the 2-category $\Icorr_2$: \begin{equation}
    \label{CorrSquare2cat}
    \stik{1}{
        (0,0) \& (M,0) \ar{r} \ar{l} \& (1,0) \\
        (0,M) \ar[Rightarrow,shorten <=1em,shorten >=1em]{dr}[above,sloped]{\sim}\ar[Rightarrow,shorten <=1em,shorten >=1em]{ur}[above,sloped]{\sim} \ar{u} \ar{d} \& (M,M)  \ar{r} \ar{l}  \ar{u} \ar{d} \& (1,M)\ar[Rightarrow,shorten <=1em,shorten >=1em]{dl}[above,sloped]{\sim} \ar[Leftarrow,shorten <=1em,shorten >=1em]{ul}[above,sloped]{\sim} \ar{u} \ar{d}\\
        (0,1) \& (M,1)  \ar{r} \ar{l} \& (1,1)
    }
\end{equation}

We want an assignment (compatible with the previous one) from any $\Icorr_2$ diagram in $\Stacks$ to an $\Icorr_2$ diagram in $\AAA$. Again, if we want to get an actual 2-morphism we need to replace some things with adjoints. Firstly, we need to take the left adjoint of the $(0,1)$ and $(1,0)$ squares, and the double left adjoint of the $(0,0)$ square to get a diagram of the form
\[
 \stik{1}{
        (0,0) \& (M,0) \ar[Rightarrow,red,shorten <=1em,shorten >=1em]{dl} \ar{r} \ar[leftarrow,red]{l} \& (1,0)  \ar[Rightarrow,red,shorten <=1em,shorten >=1em]{dl}\\
        (0,M) \ar[leftarrow,red]{u} \ar{d} \& (M,M) \ar[Leftarrow,red,shorten <=1em,shorten >=1em]{dl} \ar{r} \ar[leftarrow,red]{l}  \ar[leftarrow,red]{u} \ar{d} \& (1,M)\ar[Rightarrow,shorten <=1em,shorten >=1em]{dl}[above,sloped]{\sim} \ar[leftarrow,red]{u} \ar{d}\\
        (0,1) \& (M,1)  \ar{r} \ar[leftarrow,red]{l} \& (1,1)
    }
\]
In order to compose this we need to be able to invert the lower left morphism, or the other three. In both cases we want to end up with an invertible morphism so in fact we want to require all of them to be invertible. For the $(0,0)$ square this adds no requirement since a double adjoint of an invertible morphism is invertible, but for the $(0,1)$ and $(1,0)$ squares this needs to be a requirement on $T$. 

Note that, as shown in \autoref{Corr0} for the associator square the squares in question are pullback squares. Hence it is enough to require $T$ to be a functor from $\Stacks$ to $\AAA$ which takes pullback squares to squares satisfying the \emph{Beck-Chevalley} condition. (The Beck-Chevalley condition says precisely that the 2-morphism in the square whose sides are replaced by adjoints is invertible). 
\begin{Remark}
\label{pullbackBC}
In \cite{ourSSH} we showed that this is the same as saying that $T$ preserves a generalization of pullback squares.
\end{Remark}

This can be generalized to higher dimensional cubes using the notion of pullback cubes from \autoref{PullbackCube}. Using \autoref{Corr0} and \autoref{pullbackBC} we arrive at the following:

\begin{Definition}
\label{TransferDef}
A functor $T:\Stacks\to\AAA$ is called a \emph{transfer theory} if $T$ preserves generalized pullback cubes.
\end{Definition}

\subsection{Transfer to \texorpdfstring{$\Vect$}{Vect}}
\label{VectTransfer}
This construction is based on a functor from the 1-category of correspondences of groupoids to $\Vect$ described in \cite[\S8.2]{KapranovDyckerhoff}. We need to assume that $\CCC$ is finitary (categories of representations of a simply laced quiver satisfy this assumption).

A stack $X$ over $\field$ is sent to the vector space of finitely supported functions on the set $\pi_0(X(\field))$ (i.e. the isomorphism classes of $X(\field)$) which we denote $\Func(X)$.

Given a correspondence $X \xleftarrow{s} Y \xrightarrow{p} Z$ we first send it to the correspondence of groupoids $X(\field) \xleftarrow{s} Y(\field) \xrightarrow{p} Z(\field)$ and then to the map $\Func(X) \to \Func(Z)$ given by $p_!s^*$. This assumes some restrictions on $s$ and $p$ (see \cite[\S 2]{Dyckerhoff}) which are always satisfied in the cases we consider. 

It follows immediately from \cite[Proposition 2.17]{Dyckerhoff} that this assignment is a transfer theory to $\Vect$.

Applying this transfer to $H_{geo}$ recovers the usual Hall algebra.

\subsection{Transfer to \texorpdfstring{$\dgCat$}{dgCat}}
\label{CatTransfer}
Since all the stacks appearing in $H_{geo}$ are disjoint unions of global quotients, the assignment $X\to \ShConst(X)$ which sends a stack to the category of constructible sheaves on $X$ makes sense. i.e. if  $X\coprod X_i//G_i$ then $\ShConst(X):=\bigoplus \ShConst_{G_i}(X_i)$. The fact that the equivariant sheaves functor satisfies base change is exactly what we need in \autoref{TransferDef} when restricting to the image of $H_{geo}$.

It follows from \cite{VaragnoloVasserot} that applying this transfer to $H_{geo}$ recovers the categorification of quantum groups by KLR algebras.

\subsection{Transfer to \texorpdfstring{$\LinCat$}{LinCat}}
Define a transfer theory by sending a stack $X/\field$ to the category $\Rep_\CC(X(\field))$ of finitely supported representations of the groupoid $X(\field)$ in $\Vect_\CC$. In our article \cite{ourGLnBraiding} with Mark Penney we use this transfer in the case $\CCC=\Vect_{\FF_q}$. This yields a category equivalent to the category $\bigoplus \Rep(\GL(n,\FF_q))$ together with the monoidal structure of parabolic induction.

Using constructions from \cite{ourGeometricHall2} we use this point of view to recover the braiding constructed in \cite{JoyalStreetGLn}.

%% file: Representations.tex
\section{Representations}
\label{Representations}
In this section we outline an approach to constructing certain representations from our geometric point of view.

For simplicity we consider the case $\CCC=\Vect_{\FF_q}$. 

Let $V\in\CCC$, and consider the category $\CCC_{/V}$. As discussed in \autoref{ProtoAbelian} this is an extended proto-abelian category, and therefore we can use it as the input in the Waldhausen construction, with some caveats.

For instance for a simplex $\ordop{n}$ we get the stack of diagrams of the form

\[
\stik{1}{
0 \ar[hook]{r} \& C_{01} \ar[two heads]{d} \ar[hook]{r} \& C_{02}\ar[two heads]{d} \ar[hook]{r} \& \cdots  \ar[hook]{r} \& C_{0n}\ar[two heads]{d} \\
{} \& 0  \ar[hook]{r} \& C_{12}\ar[two heads]{d} \ar[hook]{r} \& \cdots  \ar[hook]{r} \& C_{1n}\ar[two heads]{d} \\
{} \& {} \& 0 \ar[hook]{r} \& \cdots  \ar[hook]{r} \& C_{2n}\ar[two heads]{d}\\
{} \& {} \& {} \& \ddots \& \vdots\ar[two heads]{d} \\
{} \& {} \& {} \& {} \& 0
}
\]
where in all but the last column, the map to $V$ must be the zero map. As a result, this stack has a canonical projection to the Waldhausen construction for $\CCC$ and $\ordop{n-1}$ (the first $n-1$ columns), and noting that the rest can be generated by the object $C_{0n}$ by forming pushouts, we see that in fact this stack is equivalent to $S_{\ordop{n-1}}^\CCC\times S_{\ordop{1}}^{\CCC_{/V}}$.

From the above it is obvious that not every map of simplicial sets appearing in the combinatorial Hall algebra construction will be compatible with this Waldhausen construction. However if we restrict $H_{comb}$ to the subcategory $\OrdSetMod$ of $\OrdSet$ which has only the maps where the preimage of the top element is always non-empty, then there is no problem. 

For example the map $\ord{2}\to\ord{1}$ goes to the correspondence

\[
\stik{1}{
{} \& S_{\ordop{2}}^{\CCC_{/V}} \ar{dl}[above,xshift=-1.5em]{(C_{01},C_{12})} \ar{dr}[above,xshift=0.5em]{C_{02}}\& {} \\
S_{\ordop{1}}^{\CCC}\times S_{\ordop{1}}^{\CCC_{/V}} \& {} \& S_{\ordop{1}}^{\CCC_{/V}}
}
\]
which gives the action.
The subcategory $\OrdSetMod$ embodies the structure of a module over an algebra much in the same way as $\OrdSet$ embodies the structure of an algebra. We will return to this construction in more detail in a future article.

%% file: PullbackCubes.tex
\section{Pullback cubes}
\begin{Definition}
\label{PullbackCube}
A commutative cube is said to be a \emph{pullback} cube if it presents the vertex $X=000\ldots$ as the limit of the rest of the diagram. i.e. for any other commutative cube with $\widetilde{X}=000\ldots$ there exists a unique morphism $\widetilde{X}\to X$ making everything commute.
\end{Definition}

\begin{Lemma}
\label{pullbackcubeLemma}
Consider a commutative cube in some category
\[
\stik{1}{
 \& A \ar{rr} \ar{dd} \& {} \& B \ar{dd}\\
X \ar[crossing over]{rr} \ar{dd} \ar{ur} \& {} \& Y \ar{dd} \ar{ur} \& {}\\
{} \& C \ar{rr} \& {} \& D\\
Z \ar{rr} \ar{ur} \& {} \& W \ar{ur}
\latearrow{commutative diagrams/crossing over}{2-3}{4-3}{}
}
\]
And suppose that $ABCD$ is a pullback square, then $XYZW$ is a pullback square if and only if the whole cube is a pullback cube.
\end{Lemma}

\begin{proof}
For transparency let us present a proof for a 1-category case. The appropriate generalization is straightforward. 

Assume $XYZW$ is a pullback square.

Consider another commutative cube 
\[
\stik{1}{
 \& A \ar{rr} \ar{dd} \& {} \& B \ar{dd}\\
\widetilde{X} \ar[crossing over]{rr} \ar{dd} \ar{ur} \& {} \& Y \ar{dd} \ar{ur} \& {}\\
{} \& C \ar{rr} \& {} \& D\\
Z \ar{rr} \ar{ur} \& {} \& W \ar{ur}
\latearrow{commutative diagrams/crossing over}{2-3}{4-3}{}
}
\]
We want to show that there is a unique map $\widetilde{X}\to X$ that makes everything commute.

Since $XYZW$ is a pullback square we have a unique map $\widetilde{X}\to X$ such that $\widetilde{X}Y=XY\circ\widetilde{X}X$ and $\widetilde{X}Z=XZ\circ\widetilde{X}X$.

we just need to show that $\widetilde{X}A=XA\circ\widetilde{X}X$. This follows because both sides are a map $\widetilde{X} \to A$ which make the diagram
\[
\stik{1}{
\widetilde{X} \ar{dr} \ar{drr} \ar{ddr} \\
{} \& A \ar{d} \ar{r}\& B \ar{d} \\
{} \& C \ar{r} \& D
}
\]
commute. Since we assumed $ABCD$ is a pullback such a map is unique.

Assume now that the cube is a pullback and consider a square \[
\stik{1}{
\widetilde{X} \ar{r} \ar{d} \& Y \ar{d} \\
Z \ar{r} \& W
}
\]
We want to show that there is a unique map $\widetilde{X}\to X$ which makes the diagram 
\[
\stik{1}{
\widetilde{X} \ar{dr} \ar{drr} \ar{ddr} \\
{} \& X \ar{d} \ar{r}\& Y \ar{d} \\
{} \& Z \ar{r} \& W
}
\]
commute.

The compositions $YB\circ\widetilde{X}Y$ and $ZC\circ\widetilde{X}Z$ fit in a commutative square 
\[
\stik{1}{
\widetilde{X} \ar{r} \ar{d} \& B \ar{d} \\
C \ar{r} \& D
}
\]
and so there is a map $\widetilde{X}\to A$ that makes everything commute, and since the cube is a pullback this gives us our desired map $\widetilde{X}\to X$.
\end{proof}

\begin{Corollary}
\label{pullbackcubeCor}
Let $C$ be a commutative $n$-cube. Suppose that an $n-1$-subcube $C'$ in $C$ is a pullback cube, then $C$ is a pullback iff the opposite cube to $C'$ is a pullback cube.
\end{Corollary}

\begin{proof}
Proven in the same way as \autoref{pullbackcubeLemma} by induction on $n$.
\end{proof}

%% file: ProtoAbelian.tex
\section{Proto-abelian categories}
\label{ProtoAbelian}
In \cite{Dyckerhoff} the notion of proto-abelian category is described. We briefly recall the definition here and define what we call an "extended" proto-abelian category.

\begin{Definition}
A \emph{proto-abelian category} is the data of a pointed (i.e. possessing a 0 object) category $\CCC$ together with two classes of morphisms $E,M$ ("epis" and "monos") which satisfy:
\begin{enumerate}
    \item $\CCC$ has all pushouts and pullbacks of epis along monos.
    \item pushouts or pullbacks of monos (epis) along epis (monos) are monos (epis).
    \item A square is a pushout of an epi along a mono iff it is a pulback of an epi along a mono.
\end{enumerate}
\end{Definition}

The only thing we need to change in the above definition is the pointedness of $\CCC$.

\begin{Definition}
A \emph{pseudo-zero} object of a category $\CCC$ is an object $Z$ such that we have $\#\Hom(Z,X)\leq 1$ for any $X\in\CCC$ or $\#\Hom(X,Z)\leq 1$ for any $X\in\CCC$.
\end{Definition}

\begin{Definition}
An \emph{extended} proto-abelian category is the same as a proto abelian category, except that instead of being pointed it is required to satisfy that any object has a map to and from some pseudo-zero object. We call the minimal collection of psudeo-zero objects with this property the "zero subcategory".
\end{Definition}

A morphism of extended proto-abelian categories is a functor $\CCC\to\DDD$ preserving epis and monos, their pullbacks, and zero subcategories.

\begin{Example}
The category $\grid(X)=\Hom(0\to 1,X)$ defined in \autoref{Waldhausen} is an extended proto-abelian category with zero subcategory the constant maps.
\end{Example}

\begin{Proposition}
Let $\CCC$ be a proto-abelian category, and $X\in\CCC$, then $\CCC_{/X}$ has a natural structure of an extended proto-abelian category.
\end{Proposition}

\begin{proof}
Define epis via the forgetful functor, and monos to be maps $(S\xrightarrow{0} X)\to(T\to X)$ where the underlying map $S\to T$ is a mono.

The conditions are then easy to verify, and the zero subcategory consists of $(0\to X)$ and $X\xrightarrow{\Id}X$.
\end{proof}